\documentclass[11pt]{article}
\title{Limit Theorems for Random Sums of Random Summands}
\author{David Grzybowski}
\date{}

\usepackage[margin=1.5in]{geometry}
\usepackage{amsfonts}
\usepackage{amsmath}
\usepackage{dsfont}
\usepackage[shortlabels]{enumitem}
\usepackage{etoolbox}
\usepackage{etoolbox}
\usepackage[utf8]{inputenc}
\usepackage[english]{babel}
\usepackage{amsthm}

\newtheorem{theorem}{Theorem}[section]
\newtheorem{lemma}[theorem]{Lemma}
\newtheorem{corollary}[theorem]{Corollary}

\newtheorem{definition}[theorem]{Definition}

\begin{document}
\maketitle

\begin{abstract}
We prove limit theorems for sums of randomly chosen random variables conditioned on the summands. We consider several versions of the corner growth setting, including specific cases of dependence amongst the summands and summands with heavy tails. We also prove a version of Hoeffding's combinatorial central limit theorem and results for summands taken uniformly from a random sample. These results are proved with concentration of measure techniques.
\end{abstract}

\section{Introduction and Statement of Results}

\quad This paper employs concentration of measure to prove limit theorems for sums of randomly chosen random numbers. This is a particular example of a quenched limit theorem. A quenched limit theorem involves two sources of randomness: a random environment $X$ and a random object conditioned on that environment. In our case, we have a collection of random numbers and an independent selection of a subset of those numbers. The numbers, or weights, are the components of a random $n$-dimensional vector $X$, and the random subset $\sigma$ of size $m$. Our results describe the limiting distribution of $\displaystyle\sum_{a\in\sigma}\frac{X_a-\mathbb{E}X_a}{\sqrt{m}}$, or something similar, where the weights are taken as fixed, so that randomness only comes from $\sigma$.

It is worth distinguishing our situation from the classical problem concerning the distribution of $\displaystyle\sum_{i=1}^NX_i$ where both $N$ and the $\{X_i\}_{i=1}^{\infty}$ are independent random variables. Here we are interested in the case that the summands are already sampled from their respective distributions and we do not specify a fixed order of summation. Thus, we consider the distribution of the sum conditioned on the summands and choose the summands randomly. We also take $N$ to be deterministic and consider the limiting behavior of the distribution as $N$ diverges to $\infty$.

In \textbf{[5]}, the authors address the corner-growth setting: $X$ is indexed by elements of $\Box_{N,M}=\{(i,j):1\le i\le N,1\le j\le M\}$ with $M=\lfloor\xi N\rfloor,\xi>0$. Notionally one traverses this grid from $(1,1)$ to $(N,M)$, and so $\sigma$ is chosen to represent an up-right path, i.e.\ $\sigma=\{(i_k,j_k):k=1,\dots,M+N-1\}$ with $(i_1,j_1)=(1,1)$, $(i_{N+M-1},j_{N+M-1})=(N,M)$, and $(i_{k+1}-i_k,j_{k+1}-j_k)$ is either $(1,0)$ or $(0,1)$. They provide moment conditions for a quenched limit theorem when $\sigma$ is chosen according to three different schemes:

\begin{enumerate}
\item The up-right path $\sigma$ is chosen from those proceeding from $(1,1)$ to $(N,M)$ without additional restriction.

\item As $1.$, but the up-right path $\sigma$ is also specified to pass through points $(\lfloor\zeta_i N\rfloor,\lfloor\xi_iM\rfloor)$ for a finite set of numbers, $0<\zeta_1<\cdots<\zeta_k<1$ and $0<\xi_1<\cdots<\xi_k<\xi$.

\item As $1.$, but $M=N$ and the up-right path $\sigma$ is also specified to avoid a central square with side $(\lfloor\beta N\rfloor), \beta\in(0,1)$.
\end{enumerate}

See the sources cited in \textbf{[5]} for more information on the corner growth setting.

This paper generalizes the concentration result used in \textbf{[5]} and proves limit theorems in both the corner growth setting and in other settings. In particular, that paper is concerned only with independent weights, whereas we cover certain forms of dependency amongst the weights, and we present results for weak convergence in probability as well as almost-sure convergence in distribution. We also extend to cases with heavy-tailed weights. Beyond the corner growth setting, we prove a version of Hoeffding's combinatorial central limit theorem and results related to the empirical distribution of a large random sample.

\begin{definition}Let $\{\mu_n\}_{n=1}^{\infty}$ be a sequence of random probability measures. We say $\{\mu_n\}_{n=1}^{\infty}$ \textbf{converges weakly in probability (WIP)} to $\mu$ if 
\[
\mathbb{P}\bigg[\bigg|\int fd\mu_n-\int fd\mu\bigg|>\epsilon\bigg]\to0
\]
for every $\epsilon>0$ and every test function $f\in\mathcal{C}$ where $\mathcal{C}$ is a class of functions such that if $\int fd\nu_n\to\int fd\nu$ for every $f\in\mathcal{C}$, then $\nu_n\xrightarrow{D}\nu$.
\end{definition}
We typically take $\mathcal{C}$ to be a more restricted class of functions than the bounded, continuous functions used to define convergence in distribution. In particular, we will here take $\mathcal{C}$ to be convex $1$-Lipschitz functions.\footnote{The sufficiency of convergence for convex $1$-Lipschitz functions to establish convergence in distribution can be demonstrated in several ways. One such way is to approximate power functions by convex Lipschitz functions. For $x^{2n}$, we can find an approximating function using Lemma 5 in \textbf{[13]}. For $x^{2n+1}$, we approximate $x^{2n+1}\mathds{1}_{x\ge0}$ and $x^{2n+1}\mathds{1}_{x\le0}$ separately. See \textbf{[5]}, \textbf{[6]}, and \textbf{[12]} for alternative arguments.}

Our first three results concern the corner growth setting. The following is a small extension of the theorems from \textbf{[5]} and involves both weak convergence in probability and $\mathbb{P}_X$-a.s. convergence:

\begin{theorem} Under each of the three methods for sampling $\sigma$ described above, suppose the weights are independent with mean $0$, variance $1$, and $\mathbb{E}|X_a|^p\le K<\infty$. If $p>8$, then
\[
\frac{1}{\sqrt{M+N-1}}\sum_{(i,j)\in\sigma}X_{i,j}\xrightarrow{WIP}\mathcal{N}(0,1),\text{ as }N\to\infty,
\]
and if $p>12$, then
\[
\frac{1}{\sqrt{M+N-1}}\sum_{(i,j)\in\sigma}X_{i,j}\xrightarrow{D}\mathcal{N}(0,1),\text{ as }N\to\infty,
\]
$\mathbb{P}_X$-almost surely.
\end{theorem}

The next result concerns dependent weights: the uniform distributions on the sphere and on two simplices.

\begin{theorem}Under each of the three methods for sampling $\sigma$ described above, suppose $w$ has the uniform distribution on...
\begin{enumerate}
\item$\sqrt{NM}S^{NM-1}$,

\item$\Delta_{(MN,=)}=\{x\in\mathbb{R}_+^{MN}:x_1+\cdots+x_{MN}=MN\}$, or

\item$\Delta_{(MN,\le)}=\{x\in\mathbb{R}_+^{MN}:x_1+\cdots+x_{MN}\le MN\}$.
\end{enumerate}
Then,
\[
\frac{1}{\sqrt{M+N-1}}\sum_{(i,j)\in\sigma}\big(X_{i,j}-\mathbb{E}X_{i,j}\big)\xrightarrow{D}\mathcal{N}(0,1),\text{ as }N\to\infty,
\]
$\mathbb{P}_X$-almost surely.
\end{theorem}

For the above cases, we can offer a physical interpretation: $\sqrt{NM}S^{NM-1}$ and $\Delta_{(MN,=)}$ are level-sets of specific functions of the weights, meaning that some notion of ``energy" is constant. In $\Delta_{(MN,\le)}$, this ``energy" is merely bounded. Notably, this ``energy" is constrained for the whole system, not merely for the particularly chosen summands.

Concentration is loose for stable vectors, but we can prove the following:

\begin{theorem}In each of the three situations above, suppose the weights are distributed as follows. Let $Y$ be $\alpha$-stable, $\alpha>\frac{3}{2}$, with distribution $\nu_{\alpha}$. Define $\ell_k=\{(i,j):i+j-1=k\}$. Let the components of $X$ all be independent and let all the components indexed by elements of $\ell_k$ have the same distribution as $k^{-\tau}Y$. If $\tau>2$, then
\[
\frac{1}{(2N-1)^{1/\gamma}}\sum_{(i,j)\in\sigma}X_{i,j}\xrightarrow{WIP}\nu_{\infty}
\]
and if $\tau>3$, then
\[
\frac{1}{(2N-1)^{1/\gamma}}\sum_{(i,j)\in\sigma}X_{i,j}\xrightarrow{D}\nu_{\infty}
\]
$\mathbb{P}_X$-almost surely, where $\nu_{\infty}$ has characteristic exponent
\[
\psi(t)=-\kappa^{\alpha}\displaystyle\lim_{n\to\infty}\displaystyle\sum_{k=1}^{n}\frac{k^{-\alpha\tau}}{n^{\alpha/\gamma}}|t|^{\alpha}\bigg(1-\imath\beta\operatorname{sign}(t)\tan\frac{\pi\alpha}{2}\bigg),
\]
$\kappa$ and $\beta$ are constants determined by $\nu_{\alpha}$, and $\gamma=\frac{\alpha}{1-\alpha\tau}$.
\end{theorem}

Our version of Hoeffding's combinatorial central limit theorem is as follows.

\begin{theorem}Let $X_N$ be an $N\times N$ array of independent random variables $X_{i,j},1\le i,j\le N$ with mean $0$, variance $1$, and $\mathbb{E}|X_{i,j}|^p<K$ for all $i,j$. Let $\pi$ be a permutation of $(1,\dots,N)$ chosen uniformly and $S_N=\frac{1}{\sqrt{N}}\sum_{i=1}^NX_{i,\pi(i)}$. If $p>4$, then
\[
S_N\xrightarrow{WIP}\mathcal{N}(0,1),\text{ as }N\to\infty,
\]
and if $p>6$, then
\[
S_N\xrightarrow{D}\mathcal{N}(0,1),\text{ as }N\to\infty,
\]
$\mathbb{P}_X$-almost surely.
\end{theorem}

Hoeffding proved his combinatorial central limit theorem in 1951 \textbf{[8]} for deterministic weights, and it has been refined and proved in many ways since. Random weights were apparently first considered in \textbf{[7]}. In \textbf{[2]} a version is proved with random weights using concentration of measure and Stein's method.

When we select finitely many numbers, we have the following result:

\begin{theorem}Let $X=\{X_n\}_{n=1}^{N}$ be a sample of i.i.d.\ random variables with common distribution $\mu$ such that $\mathbb{E}|X_1|^p<\infty$ for some $p>2$. If $\sigma\subset\{1,\dots,N\}$ is chosen uniformly from subsets of size $m$, with $m$ fixed, then
\[
\sum_{j\in\sigma}X_j\xrightarrow{WIP}\mu^{*m},\text{ as }N\to\infty.
\]
If $\mathbb{E}|X_1|^p<\infty$ for some $p>4$, then
\[
\sum_{j\in\sigma}X_j\xrightarrow{D}\mu^{*m},\text{ as }N\to\infty
\]
$\mathbb{P}_X$-almost surely.
\end{theorem}

When $m=1$, this is a form of the Glivenko-Cantelli theorem. In addition, we have the following result for the uniform distributions on the sphere and two simplices.

\begin{theorem}Let $X$ be uniformly distributed on $K\subset\mathbb{R}^N$ and $\sigma$ be chosen uniformly from subsets of $\{1,\dots,N\}$ of size $m$, which is fixed. If $K=\dots$
\begin{enumerate}
\item$\sqrt{N}S^{N-1}$, then
\[
m^{-1/2}\sum_{j\in\sigma}X_j\xrightarrow{D}\mathcal{N}(0,1)\text{ as }N\to\infty,
\]

\item$\Delta_{N,=}$ or $\Delta_{N,\le}$, then
\[
\sum_{j\in\sigma}X_j\xrightarrow{D}\Gamma_{m,1}\text{ as }N\to\infty,
\]
\end{enumerate}
$\mathbb{P}_X$-almost surely, where $\Gamma_{m,1}$ is the Gamma distribution with shape paramter $m$ and scale parameter $1$.
\end{theorem}

\section{Concentration and Convergence}

\subsection{General Concentration}

\quad We employ a classical result from Talagrand \textbf{[16]} as stated in \textbf{[15]}:

\begin{theorem}Let $X=(X_1,\dots,X_n)$ be a random vector with independent components such that for all $1\le i\le n$, $|X_i|\le1$ almost surely, and let $f:\mathbb{R}\to\mathbb{R}$ be a convex $1$-Lipschitz function. Then for all $t>0$
\[
\mathbb{P}[|f(X)-\mathbb{E}f(X)|>t]\le Ce^{-ct^2}
\]
where $C,c>0$ are absolute constants.
\end{theorem}

Talagrand's theorem is an example of subgaussian concentration (SGC):
\[
\mathbb{P}\big[|f(X)-\mathbb{E}f(X)|>t\big]\le Ce^{-ct^2}
\]
for a specified class of $1$-Lipschitz functions $f$. $C$ and $c$ are not dependent on dimension, but they do vary depending on the distribution of $X$. Examples of distributions for $X$ satisfying SGC include i.i.d.\ components satisfying a log-Sobolev inequality, independent bounded components (for convex $1$-Lipschitz functions, per Theorem 2.1), and the uniform distribution on a sphere. As this last example shows, this property does not require independence among the components of $X$, but it does require all their moments to be finite.

We will use SGC in the form of Theorem 2.1 as a component in our general concentration result and to prove a similar result for dependent weights. In addition, we also have subexponential concentration (SEC). The inequality becomes:
\[
\mathbb{P}\big[|f(X)-\mathbb{E}f(X)|>t\big]\le Ce^{-ct},
\]
for a specified class of $1$-Lipschitz function $f$. SEC is proved for the two simplices $\Delta_{(N,=)}$ and $\Delta_{(N,\le)}$ in \textbf{[1]} and \textbf{[14]}.

We now present our new concentration lemma, which generalizes the result from \textbf{[5]}. As setting, suppose $\Sigma$ is a set with $n$ elements equipped with independent weights $\{X_{a}:a\in\Sigma\}$ with mean zero and $\mathbb{E}|X_a|^p\le K<\infty$ for some $p>1$. Let $\sigma$ be a random subset of $\Sigma$ with $m$ elements chosen independently of $X$. We will specify the distribution of $\sigma$ in applications. $\mathbb{P}_X$ and $\mathbb{P}_{\sigma}$ are the respective marginals. For a test function $f:\mathbb{R}\to\mathbb{R}$ we set
\[
\int fd\mu_X=\mathbb{E}_{\sigma}f\bigg(m^{-1/\alpha}\sum_{a\in\sigma}X_{a}\bigg)
\]
for some $0<\alpha\le\infty$. Also set $L=\bigg(\sum_{a\in\Sigma}\mathbb{P}_{\sigma}(a\in\sigma)^2\bigg)^{1/2}$. Finally, recall the $1$-Wasserstein distance between probability measures $\mu$ and $\nu$:
\[
d_W(\mu,\nu)=\sup_{|f|_L\le1}\bigg|\int fd\mu-\int fd\nu\bigg|,
\]
where $|f|_L$ is the Lipschitz constant of $f$.

\begin{lemma}In the setting described above, there exist absolute constants $C,c>0$ such that for any $s,t,R,0<\alpha\le\infty$, any probability measure $\nu$, and any convex $1$-Lipschitz function $f:\mathbb{R}\to\mathbb{R}$,
\begin{enumerate}
\item if $1\le p<2$, then
\begin{align*}
\mathbb{P}_X\bigg[\bigg|\int fd\mu_X-\int fd\nu\bigg|\ge D+\frac{LKn}{m^{1/\alpha}R^{p-1}}+s+t\biggr]\le\frac{LKn}{m^{1/\alpha}R^{p-1}s}+\\C\exp\bigg[-c\frac{m^{2/\alpha}t^2}{L^2R^2}\bigg],
\end{align*}
and
\item if $\mathbb{E}X_a^2=1$ for all $a$ and $2\le p$, then
\begin{align*}
\mathbb{P}_X\bigg[\bigg|\int fd\mu_X-\int fd\nu\bigg|\ge D+\frac{L\sqrt{Kn}}{m^{1/\alpha}\sqrt{R^{p-2}}}+s+t\bigg]\le\frac{L^2Kn}{m^{2/\alpha}R^{p-2}s^2}+\\C\exp\bigg[-c\frac{m^{2/\alpha}t^2}{L^2R^2}\bigg].
\end{align*}
\end{enumerate}
where $D=\displaystyle\max_{\sigma}d_W(\rho_{\sigma},\nu)$ and $\rho_{\sigma}$ is the distribution of $m^{-1/\alpha}\sum_{a\in\sigma}X_a$ conditioned on $\sigma$.
\end{lemma}

\begin{proof}First assume $\alpha<\infty$. We take the same approach as in the proof of the concentration lemma of \textbf{[5]}. For a fixed $R>0$, we define the truncations $X_a^{(R)}=X_a\mathds{1}_{|X_a|\le R}$ and denote the distribution of $m^{-1/\alpha} \sum_{a \in \sigma} X_a^{(R)}$ conditioned on $X$ as $\mu_X^{(R)}$. We split the integral
\begin{align*}
\bigg|\int fd\mu_X-\int fd\nu\bigg|\le&\bigg|\int fd\mu_X-\int fd\mu_X^{(R)}\bigg|\tag{2.1}\\
&+\bigg|\int fd\mu_X^{(R)}-\mathbb{E}_X\int fd\mu_X^{(R)}\bigg|\tag{2.2}\\
&+\bigg|\mathbb{E}_X\int fd\mu_X^{(R)}-\mathbb{E}_X\int fd\mu_X\bigg|\tag{2.3}\\
&+\bigg|\mathbb{E}_X\int fd\mu_X-\int fd\nu\bigg|.\tag{2.4}
\end{align*}
Each of these items can be bounded individually, either absolutely or with high probability. The easiest is (2.4), which is bounded absolutely by Fubini's Theorem.
\begin{align*}
\bigg|\mathbb{E}_X\int fd\mu_X-\int fd\nu\bigg|=&\bigg|\mathbb{E}_X\mathbb{E}_{\sigma}f\bigg(m^{-1/\alpha}\sum_{a\in\sigma}X_a\bigg)-\int fd\nu\bigg|\\
\le&\mathbb{E}_{\sigma}\bigg|\mathbb{E}_Xf\bigg(m^{-1/\alpha}\sum_{a\in\sigma}X_a\bigg)-\int fd\nu\bigg|\\
\le&\max_{\sigma}d_W(\rho_{\sigma},\nu).
\end{align*}
(2.2) is next. As in \textbf{[5]}, we can set
\[
F(X)=\int fd\mu_X=\mathbb{E}_{\sigma}f\bigg(m^{-1/\alpha}\sum_{a\in\sigma}X_a\bigg),
\]
which is convex and Lipschitz. For $X,X'\in\mathbb{R}^{\Sigma}$,
\begin{align*}
|F(X)-F(X')|\le&\mathbb{E}_{\sigma}\bigg|f\bigg(m^{-1/\alpha}\sum_{a\in\sigma}X_a\bigg)-f\bigg(m^{-1/\alpha}\sum_{a\in\sigma}X'_a\bigg)\bigg|\\
\le&m^{-1/\alpha}\mathbb{E}_{\sigma}\bigg|\sum_{a\in\sigma}X_a-\sum_{a\in\sigma}X'_a\bigg|\\
\le&m^{-1/\alpha}\mathbb{E}_{\sigma}\sum_{a\in\sigma}|X_a-X'_a|\\
\le&m^{-1/\alpha}L||X-X'||_2,
\end{align*}
so $F$ is $\frac{L}{m^{1/\alpha}}$-Lipschitz. Applying Theorem 2.1 gives
\[
\mathbb{P}_X\bigg[\bigg|\int fd\mu_X^{(R)}-\mathbb{E}_X\int fd\mu_X^{(R)}\bigg|\ge t\bigg]\le C\exp\bigg[-c\frac{m^{2/\alpha}t^2}{L^2R^2}\bigg].
\]
For the other terms, the situation differs depending on case $1.$ or case $2.$ In either case, we use the Lipschitz estimate
\begin{align*}
\bigg|\int fd\mu_X-\int fd\mu_X^{(R)}\bigg|\le&\frac{L}{m^{1/\alpha}}||X-X^{(R)}||_2\\
=&\frac{L}{m^{1/\alpha}}\bigg(\sum_{a\in\Sigma}X_a^2\mathds{1}_{|X_a|>R}\bigg)^{1/2}.
\end{align*}
In case $2.$, we can use H\"older's and Chebyshev's inequalities to give
\begin{align*}
\mathbb{E}_X\bigg(\sum_{a\in\Sigma}X_a^2\mathds{1}_{|X_a|>R}\bigg)^{1/2}\le&\big(\mathbb{E}_X|X_a|^p\big)^{2/p}\big(\mathbb{P}_X[|X_a|>R]\big)^{1-2/p}\\
\le&\frac{\mathbb{E}_X|X_a|^p}{R^{p(1-2/p)}}\\
\le&\frac{KL}{R^{p-2}}.
\end{align*}
By Markov's and Chebyshev's inequalities, then
\[
\mathbb{P}_X\bigg[\sum_{a\in\Sigma}X_a^2\mathds{1}_{|X_a|>R}\ge u\bigg]\le\frac{nK}{uR^{p-2}},
\]
from which we have an absolute bound on (2.3)
\[
\bigg|\mathbb{E}_X\int fd\mu_X-\mathbb{E}_X\int fd\mu_X^{(R)}\bigg|\le\frac{L\sqrt{Kn}}{m^{1/\alpha}\sqrt{R^{p-2}}}
\]
and a high-probability bound on (2.1)
\[
\mathbb{P}_X\bigg[\bigg|\int fd\mu_X-\int fd\mu_X^{(R)}\bigg|\ge s\bigg]\le\frac{nKL^2}{m^{2/\alpha}R^{p-2}s^2}.
\]
Combining these bounds gives the result.

In case $1.$, the weights are not guaranteed to have finite variance, so we need to use
\[
\bigg(\sum_{a\in\Sigma}X_a^2\mathds{1}_{|X_a|>R}\bigg)^{1/2}\le\sum_{a\in\Sigma}|X_a|\mathds{1}_{|X_a|>R},
\]
and again by H\"older's and Chebyshev's inequalities,
\begin{align*}
\mathbb{E}_X|X_a|\mathds{1}_{|X_a|>R}\le&(\mathbb{E}_X|X_a|^p)^{1/p}\mathbb{P}_X(|X_a|>R)^{(p-1)/p}\\
\le&(\mathbb{E}_X|X_a|^p)^{1/p}\bigg(\frac{\mathbb{E}_X|X_a|^p}{R^p}\bigg)^{(p-1)/p}\\
=&\frac{\mathbb{E}_X|X_a|^p}{R^{p-1}}\\
=&\frac{K}{R^{p-1}}.
\end{align*}
From this result and Markov's inequality,
\[
\mathbb{P}_X\bigg[\sum_{a\in\Sigma}|X_a|\mathds{1}_{|X_a|>R}\ge u\bigg]\le\frac{Kn}{uR^{p-1}},
\]
so that
\[
\bigg|\mathbb{E}_X\int fd\mu_X-\mathbb{E}_X\int fd\mu_X^{(R)}\bigg|\le\frac{LKn}{m^{1/\alpha}R^{p-1}}
\]
and
\[
\mathbb{P}_X\bigg[\bigg|\int fd\mu_X-\int fd\mu_X^{(R)}\bigg|\ge s\bigg]\le\frac{LKn}{m^{1/\alpha}R^{p-1}s}.
\]
For $\alpha=\infty$, this corresponds to the case of the sums not being rescaled at all. The result in this case follows by observing that removing the term $m^{-1/\alpha}$ altogether does not change the validity of the above argument.
\end{proof}

The next lemma covers the SGC and SEC cases. For this we recall the Bounded Lipschitz distance between two probability measures $\mu$ and $\nu$:
\[
d_{BL}(\mu,\nu)=\sup_{|f|_{BL}\le1}\bigg|\int fd\mu-\int fd\nu\bigg|,
\]
where $|f|_{BL}=\max\{||f||_{\infty},|f|_L\}$. It is worth noting that the Bounded Lipschitz distance metrizes weak convergence.

\begin{lemma}Let the assumptions and notation be as in Lemma 2.2, except that $X$ no longer need have independent components, $f$ no longer need be convex, and $|f|_{BL}\le1$. If for bounded Lipschitz functions $X$ has, with constants $C$ and $c$,...
\begin{enumerate}
\item SGC, then
\[
\mathbb{P}_X\bigg[\bigg|\int fd\mu_X-\int fd\nu\bigg|>E+t\bigg]\le C\exp\bigg(-c\frac{m^{2/\alpha}t^2}{L^2}\bigg),
\]

\item SEC, then
\[
\mathbb{P}_X\bigg[\bigg|\int fd\mu_X-\int fd\nu\bigg|>E+t\bigg]\le C\exp\bigg(-c\frac{m^{1/\alpha}t}{L}\bigg),
\]
\end{enumerate}
where $E=\displaystyle\max_{\sigma}d_{BL}(\rho_{\sigma},\nu)$ and $\rho_{\sigma}$ is the distribution of $m^{-1/\alpha}\sum_{a\in\sigma}X_a$ conditioned on $\sigma$.
\end{lemma}

\begin{proof}The proof is the same as for Lemma 2.2 except that the truncation step is not necessary. We have
\begin{align*}
\bigg|\int fd\mu_X-\int fd\nu\bigg|\le&\bigg|\int fd\mu_X-\mathbb{E}_X\int fd\mu_X\bigg|+\bigg|\mathbb{E}_X\int fd\mu_X-\int fd\nu\bigg|.
\end{align*}
The first term is bounded with high probability by SGC or SEC, and the second is bounded according to Fubini's theorem. Since we have specified $|f|_{BL}\le1$, we can use the Bounded Lipschitz distance instead of the Wasserstein distance.
\end{proof}

For a further extension, let us first recall the definition of stable distributions.

\begin{definition}A random variable $X$ has an \textbf{$\alpha$-stable distribution} $\nu_{\alpha},0<\alpha<2$, if its characteristic function $\phi_X(t)=\mathbb{E}e^{\imath Xt}=e^{\psi(t)}$ where
\[
\psi(t)=-\kappa^{\alpha}|t|^{\alpha}\big(1-\imath\beta\operatorname{sign}(t)\tan\frac{\pi\alpha}{2}\big).\tag{2.5}
\]
for some $\beta\in[-1,1]$ and $\kappa>0$. For a $d$-dimensional \textbf{$\alpha$-stable random vector}, the characteristic exponent is
\[
\psi(t)=-t\int_{S^{d-1}}\int_0^{\infty}d\lambda(\xi)\mathds{1}_{B(r\xi,0)}e^{\imath tx}-1-\imath tx\mathds{1}_{|t|<1}r^{-(\alpha+1)}drd\lambda(\xi)\tag{2.6}
\]
with $\lambda$ a finite poisitive measure on $S^{d-1}$. The double integral can be rewritten as a single integral with respect to a measure called the L\'evy measure.
\end{definition}

Case $1.$ of Lemma 2.2 covers stable weights and weights in the domain of attraction of a stable distribution, but unfortunately, it is insufficient for proving convergence of any kind, nor are most concentration results for infinitely divisible vectors any better. The following is one useful result, however, from \textbf{[11]}.

\begin{theorem}For $\alpha>\frac{3}{2}$, let $X$ be a $d$-dimensional $\alpha$-stable random vector with characteristic exponent as in Definition 2.4. For any $1$-Lipschitz function $f:\mathbb{R}^d\to\mathbb{R}$, 
\[
\mathbb{P}[|f(X)-\mathbb{E}f(X)|\ge x]\le\frac{K\lambda\big(S^{d-1}\big)}{x^{\alpha}}
\]
for every $x$ such that
\[
x^{\alpha}\ge K_{\alpha}\lambda\big(S^{d-1}\big)
\]
where $K$ is an absolute constant and $K_{\alpha}$ is a constant depending only on $\alpha$.
\end{theorem}

Further concentration results can be found in \textbf{[9]}, \textbf{[10]}, and \textbf{[11]}. From this concentration inequality, we can prove a result leading to Theorem 1.4. Notice that here we distinguish between the index of stability for the vectors and the exponent used in rescaling the sum.

\begin{lemma}With the same assumptions and notation as in Lemma 2.2 except that $X$ is an $\alpha'$-stable random vector with $\alpha'>\frac{3}{2}, d=n$ and L\'evy measure as in (2.6) and $\alpha<\infty$. For any $1$-Lipschitz function $f:\mathbb{R}\to\mathbb{R}$,
\[
\mathbb{P}_X\bigg[\bigg|\int fd\mu_X-\int fd\nu\bigg|>D+t\bigg]\le\frac{KL^{\alpha}\lambda\big(S^{d-1}\big)}{mt^{\alpha}}
\]
for any $t$ satisfying
\[
t^{\alpha'}\ge L^{\alpha}K_{\alpha'}\lambda\big(S^{d-1}\big)m^{-1/\alpha}.
\]
\end{lemma}

\begin{proof}The proof is the same as for Lemma 2.3 except using Theorem 2.5 instead of SGC.
\end{proof}

\subsection{Convergence Conditions}

\quad In this section, we apply the concentration results of Lemmas 2.2, 2.3, and 2.6 to prove convergence results from which the main results will follow as simple corollaries. For this purpose, we assume an infinite family of random vectors $X_N$ with entries indexed by finite index sets $\Sigma_N$, from which we take subsets $\sigma_N\subset\Sigma_N$, all in turn indexed by $N\in\mathbb{N}$.

We will assume here $n=n(N)=|\Sigma_N|\sim N^{\eta}, m=m(N)=|\sigma_N|\sim N^{\mu},$ and $L=L(N)\sim N^{\lambda}$ where $\eta>0,\mu\ge0,$ and $\lambda$ are constants. Our result establishes the necessary relationships between $\alpha, p, \eta, \mu$ and $\lambda$ for convergence theorems. The truncation parameter $R$ offers a measure of freedom, so we will assume $R(N)\sim N^{\rho},\rho\ge0$. Naturally, $\mu\le\eta$.

\begin{lemma}Let the setting be as in Lemma 2.2 and described above. For $\alpha<\infty$, if we have $\lambda<\frac{\mu}{\alpha}$ and $p>\frac{\alpha\eta}{\mu-\alpha\lambda}$, then
\[
\mathbb{P}_{X_N}\bigg[\bigg|\int fd\mu_{X_N}-\int fd\nu\bigg|\ge D_N+o(1)\bigg]\to0,\text{ as }N\to\infty,
\]
where $D_N$ is the value $D$ in Lemma 2.2 for $X_N$, and if additionally $p>\frac{\alpha(\eta+1)}{\mu-\alpha\lambda}$, then
\[
\sum_{N=1}^{\infty}\mathbb{P}_{X_N}\bigg[\bigg|\int fd\mu_{X_N}-\int fd\nu\bigg|>D_{N}+M_{N,i}+o(1)\bigg]<\infty,
\]
where $i=1,2$ stands for the case in Lemma 2.2: in case $1.$, $M_{N,1}=\frac{L(N)Kn(N)}{m(N)^{1/\alpha}R(N)^{p-1}}$ and in case $2.$, $M_{N,2}=\frac{L(N)\sqrt{Kn(N)}}{m(N)^{1/\alpha}\sqrt{R(N)^{p-2}}}$.

The same results apply when $\alpha=\infty$ with moment conditions $p>\frac{\eta}{-\lambda}$ and $p>\frac{\eta+1}{-\lambda}$ provided $0>\lambda$.
\end{lemma}

\begin{proof}In case $1.$, take $s=o(1)$, $t=o(1)$, and $R\sim N^{\rho}$. Using Lemma 2.2, the first convergence requires
\[
\lambda+\eta-\frac{\mu}{\alpha}-\rho(p-1)<0\tag{2.10}
\]
and
\[
\frac{2\mu}{\alpha}-2\lambda-2\rho>0.\tag{2.11}
\]
The second condition implies $\lambda<\frac{\mu}{\alpha}$. Taking $\rho=\frac{\mu}{\alpha}-\lambda-\epsilon>0$ with $\epsilon>0$ gives $p>\frac{\alpha(\eta-\epsilon)}{\mu-\alpha\lambda-\alpha\epsilon}$. We pass to the limit $\epsilon\to0^+$ to give the appropriate lower bound for $p$.

For summability, the first condition becomes
\[
\lambda+\eta-\frac{\mu}{\alpha}-\rho(p-1)<-1,\tag{2.12}
\]
which now gives $p>\frac{\alpha(\eta+1)}{\mu-\alpha\lambda}$.

In case $2.$, equation (2.11) is the same, but (2.10) and (2.12) are replaced with
\[
2\lambda+\eta-\frac{2\mu}{\alpha}-\rho(p-2)<0\text{ or }-1,
\]
respectively, which, interestingly, give the same conditions as in case $1.$

For $\alpha=\infty$, the same process gives the result, simply removing the $m$ terms. The condition on $\lambda$ is required so that $\rho>0$.
\end{proof}

When Lemma 2.3 applies, convergence is easier.

\begin{lemma}Let the setting be as in Lemma 2.7 except that $f$ need not be convex and $|f|_{BL}\le1$ and $X_N$ has SGC or SEC for all $N$. If $\frac{\mu}{\alpha}>\lambda$, then
\[
\mathbb{P}_{X_N}\bigg[\bigg|\int fd\mu_{X_N}-\int fd\nu\bigg|\ge E_N+o(1)\bigg]\to0,\text{ as }N\to\infty,
\]
and
\[
\sum_{N=1}^{\infty}\mathbb{P}_{X_N}\bigg[\bigg|\int fd\mu_{X_N}-\int fd\nu\bigg|>E_N+o(1)\bigg]<\infty.
\]
\end{lemma}

\begin{proof}This follows from Lemmas 2.3 in the same way as Lemma 2.7 follows from Lemma 2.2.
\end{proof}

Finally, we have a convergence result from Lemma 2.6.

\begin{lemma}Let the setting be as in Lemma 2.4 and the vectors $X_N$ be $\alpha'$-stable with $\alpha'>\frac{3}{2}$ and characteristic exponents as in Definition 2.4. Suppose further $\lambda_N(S^{n(N)-1})=O\big(N^{-\tau}\big),\tau>0$. If $\alpha\lambda-\tau-\mu<0$, then
\[
\mathbb{P}_{X_N}\bigg[\bigg|\int fd\mu_{X_N}-\int fd\nu\bigg|\ge D_N+o(1)\bigg]\to0,\text{ as }N\to\infty
\]
and if $\alpha\lambda-\tau-\mu<-1$, then
\[
\sum_{N=1}^{\infty}\mathbb{P}_{X_N}\bigg[\bigg|\int fd\mu_{X_N}-\int fd\nu\bigg|>D_N+o(1)\bigg]<\infty.
\]
\end{lemma}

\begin{proof}This follows from Lemma 2.6 in the same way that Lemma 2.7 follows from Lemma 2.3.
\end{proof}

\section{Proofs of Main Results}

\quad Proving the main theorems requires only a little bit more than the results from the previous section. The first desideratum is for $\displaystyle\limsup_{N\to\infty}D_N=0$ or $\displaystyle\limsup_{N\to\infty}E_N=0$, which is guaranteed by appropriate choice of target measure $\nu$. The second is a bound on $\lambda$, which we recall to be such that $L=L(N)\sim N^{\lambda}$. In Section 3 of \textbf{[5]}, the authors prove the following:

\begin{lemma}In the corner growth setting, in each of the three cases listed in Section 1, $\lambda<\frac{1}{4}$.
\end{lemma}

This allows us to prove all our results in the corner growth setting.

\begin{proof}[Proof of Theorems 1.2, 1.3, \&, 1.4] Theorem 1.2 follows from Lemma 2.7, the Borel-Cantelli lemma, Lemma 3.1, and a central limit theorem along the lines of \textbf{[5]}. Notice that such a theorem requires bounded $3^{\text{rd}}$ absolute moments, which are guaranteed by our conditions.

Theorem 1.3 for the sphere follows from Lemma 2.8 (on account of SGC), Lemma 3.1, and the following reasoning. For $1.$, follows from a result from \textbf{[3]}: the authors prove a convergence result in the total variation distance (Inequality (1))
\[
d_{TV}(\mathcal{L}(X_{\sigma}),\mathcal{L}(Z))\le\frac{2(N+M+2)}{NM-N-M-2},\tag{3.1}
\]
where $X_{\sigma}$ is the vector of weights indexed by $\sigma$ and $Z$ is a vector of i.i.d. standard Gaussian random variables. Since $d_{BL}(\mu,\nu)\le d_{TV}(\mu,\nu)$, the triangle inequality gives the following,
\begin{align*}
\bigg|\int fd\mu_X-\int fd\nu\bigg|\le&\bigg|\int fd\mu_X-\int f\bigg((N+M-1)^{-1/2}\bigg(\sum_{i=1}^{N+M-1}Z_i\bigg)\bigg)d\mathcal{L}(Z)\bigg|-\\
&\bigg|\int f\bigg((N+M-1)^{-1/2}\bigg(\sum_{i=1}^{N+M-1}Z_i\bigg)\bigg)d\mathcal{L}(Z)-\int fd\nu\bigg|\\
\le&d_{BL}(\mathcal{L}(X_{\sigma}),\mathcal{L}(Z))+d_{BL}(\mathcal{L}((N+M-1)^{-1/2}(Z_1+\cdots+Z_{N+M-1}),\nu).\tag{3.2}
\end{align*}
where $\nu\sim\mathcal{N}(0,1)$. The second term on the right-hand side of (3.2) is $0$, so (3.1) gives the result. For $2.$, the proof is the same except that we use (3.4) from \textbf{[3]}, which gives a similar bound on the total variation distance between the coordinates of the simplex and a vector of i.i.d. $1$-exponential random variables, and $d_{BL}(\mathcal{L}((N+M-1)^{-1/2}(x_1+\cdots+x_{N+M-1}),\nu)\to0$ by the classical central limit theorem. Similar reasoning gives the result for $3.$

Theorem 1.4 follows from Lemma 2.9, the Borel-Cantelli lemma, Lemma 3.1, and the following reasoning. As a preliminary, observe that $\displaystyle\sum_{k=1}^Nk^{-\gamma}=O(N^{1-\gamma})$. Next, observe that by construction $(2N-1)^{-1/\gamma}\displaystyle\sum_{(i,j)\in\sigma}X_{(i,j)}$ has the same distribution, $\nu_N$, regardless of $\sigma$ and that distribution converges to $\nu_{\infty}$, again by construction. To confirm $d_W(\nu_N,\nu_{\infty})$ converges to $0$, for any $1$-Lipschitz function $f:\mathbb{R}\to\mathbb{R}$,
\begin{align*}
\bigg|\int fd\nu_N-\int fd\nu_{\infty}\bigg|\le&\mathbb{E}||X_N-X_{\infty}||\\
\le&\mathbb{E}|X|\bigg(\lim_{n\to\infty}\sum_{k=1}^{n}\frac{k^{-\tau}}{n^{1/\gamma}}-\sum_{k=1}^{n}\frac{k^{-\tau}}{n^{1/\gamma}}\bigg),
\end{align*}
where $X_N$ and $X_{\infty}$ have distribution $\nu_N$ and $\nu_{\infty}$ respectively.

The final step is to show that under the specified conditions, $\lambda(S^{2N-2})$ satisfies Lemma 2.9. By construction,
\begin{align*}
\lambda(S^{2N-2})=&1+\sum_{n=2}^N\sum_{k=N}^{2N-2}2k^{-\tau}+(2N-1)^{-\tau}\\
\le&2\sum_{n=1}^NN^{1-\tau}\\
\le&2N^{2-\tau},
\end{align*}
so the conditions satisfy Lemma 2.9.
\end{proof}

In the case of $\sigma$ being uniformly distributed over subsets of size $m$, we can compute $\lambda$ exactly.

\begin{lemma}When $\sigma$ is chosen uniformly from subsets of size $m$, $\lambda=\mu-\frac{\eta}{2}$.
\end{lemma}

\begin{proof}In this case, the probability that any particular element of $\Sigma$ is in $\sigma$ is $\frac{\binom{n-1}{m-1}}{\binom{n}{m}}=\frac{m}{n}$, so $L^2=\frac{m^2}{n}\sim N^{2\mu-\eta}$. Thus, $\lambda=\mu-\frac{\eta}{2}$.
\end{proof}

\begin{proof}[Proof of Theorems 1.6 \& 1.7] Theorem 1.6 follows from Lemma 2.7, the Borel-Cantelli lemma, and Lemma 3.2, taking $\alpha=\infty,\eta=1$, and $\mu=0$, since by construction $D=0$.

Theorem 1.7 follows from Lemma 2.8, Lemma 3.2, and the same reasoning as in the proof of Theorem 1.3. For the case of the simplex, recall that the $m$ entries in the sum are approximately independent $1$-exponential random variables, the sum of which has a Gamma distribution with shape parameter $m$ and scale parameter $1$.
\end{proof}

Another simple corollary of Lemma 3.2 vaguely connected to the corner growth setting comes from taking $\alpha=\eta=2$, and $\mu=1$.

\begin{corollary}Let the setting be as in Lemma 2.4 with $\alpha=\eta=2$ and $\mu=1$, and $\sigma$ is chosen uniformly from subsets of $\Sigma$ of size $m$. If $p>4$, then
\[
\frac{1}{\sqrt{M+N-1}}\sum_{(i,j)\in\sigma}X_{i,j}\xrightarrow{WIP}\mathcal{N}(0,1),\text{ as }N\to\infty,
\]
and if $p>6$, then
\[
\frac{1}{\sqrt{M+N-1}}\sum_{(i,j)\in\sigma}X_{i,j}\xrightarrow{D}\mathcal{N}(0,1),\text{ as }N\to\infty,
\]
$\mathbb{P}_X$-almost surely.
\end{corollary}

In the setting for the combinatorial central limit theorem, we again can compute $\lambda$ exactly.

\begin{lemma}When $X$ is an $N\times N$ rectangular array of numbers, so that $\Sigma=\{(i,j):1\le i,j\le N\}$, with $\sigma$ chosen uniformly from subsets each containing exactly one number from each row, $\lambda=0$.
\end{lemma}

\begin{proof}For fixed $i$, the probability that $(i,j)\in\sigma$ is exactly the same regardless of $j$, so it is $\frac{1}{N}$. Thus, $L^2(N)=\big(\frac{1}{N^2}\big)N^2=1$, so $\lambda=0$.
\end{proof}

\begin{proof}[Proof of Theorem 1.5] Take $\alpha=\eta=2$ and $\mu=1$. The result follows from Lemma 2.7, the Borel-Cantelli lemma, and Lemma 3.4.
\end{proof}

\begin{center}\textbf{Acknowledgements}\end{center}

This paper is part of the author's Ph.D. dissertation written under the supervision of Mark W. Meckes. The author would like to thank Professor Meckes for his helpful comments and discussions and the anonymous referee for suggestions that improved the exposition.

\textit{Email address}: drg67@case.edu

DEPARTMENT OF MATHEMATICS, APPLIED MATHEMATICS, AND STATISTICS, CASE WESTERN RESERVE UNIVERSITY, 10900 EUCLID AVE. CLEVELAND, OH 44106, U.S.A.

\end{document}